\documentclass[12pt]{article}
\usepackage{amssymb,amsmath,amsthm,latexsym,hyperref}
\allowdisplaybreaks

\usepackage{color}

\usepackage[a4paper,margin=1in]{geometry} 
\usepackage{verbatim}
\usepackage{dsfont}
\usepackage{marginnote}

\begin{document}

\begin{doublespace}
\def\1{{\bf 1}}
\def\ind{{\bf 1}}
\def\nn{\nonumber}
\newcommand{\I}{\mathbf{1}}

\def\sA {{\cal A}} \def\sB {{\cal B}} \def\sC {{\cal C}}
\def\sD {{\cal D}} \def\sE {{\cal E}} \def\sF {{\cal F}}
\def\sG {{\cal G}} \def\sH {{\cal H}} \def\sI {{\cal I}}
\def\sJ {{\cal J}} \def\sK {{\cal K}} \def\sL {{\cal L}}
\def\sM {{\cal M}} \def\sN {{\cal N}} \def\sO {{\cal O}}
\def\sP {{\cal P}} \def\sQ {{\cal Q}} \def\sR {{\cal R}}
\def\sS {{\cal S}} \def\sT {{\cal T}} \def\sU {{\cal U}}
\def\sV {{\cal V}} \def\sW {{\cal W}} \def\sX {{\cal X}}
\def\sY {{\cal Y}} \def\sZ {{\cal Z}}

\def\bA {{\mathbb A}} \def\bB {{\mathbb B}} \def\bC {{\mathbb C}}
\def\bD {{\mathbb D}} \def\bE {{\mathbb E}} \def\bF {{\mathbb F}}
\def\bG {{\mathbb G}} \def\bH {{\mathbb H}} \def\bI {{\mathbb I}}
\def\bJ {{\mathbb J}} \def\bK {{\mathbb K}} \def\bL {{\mathbb L}}
\def\bM {{\mathbb M}} \def\bN {{\mathbb N}} \def\bO {{\mathbb O}}
\def\bP {{\mathbb P}} \def\bQ {{\mathbb Q}} \def\bR {{\mathbb R}}
\def\bS {{\mathbb S}} \def\bT {{\mathbb T}} \def\bU {{\mathbb U}}
\def\bV {{\mathbb V}} \def\bW {{\mathbb W}} \def\bX {{\mathbb X}}
\def\bY {{\mathbb Y}} \def\bZ {{\mathbb Z}}
\def\R {{\mathbb R}} \def\RR {{\mathbb R}} \def\H {{\mathbb H}}
\def\n{{\bf n}} \def\Z {{\mathbb Z}}

\newcommand{\expr}[1]{\left( #1 \right)}
\newcommand{\cl}[1]{\overline{#1}}
\newtheorem{thm}{Theorem}[section]
\newtheorem{lemma}[thm]{Lemma}
\newtheorem{defn}[thm]{Definition}
\newtheorem{prop}[thm]{Proposition}
\newtheorem{corollary}[thm]{Corollary}
\newtheorem{remark}[thm]{Remark}
\newtheorem{example}[thm]{Example}
\numberwithin{equation}{section}
\def\ee{\varepsilon}
\def\qed{{\hfill $\Box$ \bigskip}}
\def\NN{{\mathcal N}}
\def\AA{{\mathcal A}}
\def\MM{{\mathcal M}}
\def\BB{{\mathcal B}}
\def\CC{{\mathcal C}}
\def\LL{{\mathcal L}}
\def\DD{{\mathcal D}}
\def\bD{{\mathbb D}}
\def\FF{{\mathcal F}}
\def\EE{{\mathcal E}}
\def\HH{{\mathcal H}}
\def\QQ{{\mathcal Q}}
\def\SS{{\mathcal S}}
\newcommand\II{\mathcal{I}}
\def\RR{{\mathbb R}}
\def\R{{\mathbb R}}
\def\L{{\bf L}}
\def\K{{\bf K}}
\def\S{{\bf S}}
\def\A{{\bf A}}
\def\E{{\mathbb E}}
\def\bG{{\mathbb G}}
\def\F{{\bf F}}
\def\P{{\mathbb P}}
\def\N{{\mathbb N}}
\def\eps{\varepsilon}
\def\wh{\widehat}
\def\wt{\widetilde}
\def\pf{\noindent{\bf Proof.} }
\def\pff{\noindent{\bf Proof} }
\def\cp{\mathrm{Cap}}

\newcommand{\dc}{\mathfrak{A}}
\newcommand{\X}{\mathfrak{X}}
\newcommand{\dom}{\mathcal{D}}
\newcommand{\pr}{\mathbf{P}}
\newcommand{\ex}{\mathbf{E}}
\newcommand{\M}{\mathcal{M}}
\newcommand{\sub}{\subseteq}
\newcommand{\ph}{\varphi}
\newcommand{\ro}{\varrho}
\newcommand{\pot}{\mathcal{U}}
\newcommand{\potY}{\mathcal{V}}
\newcommand{\norm}[1]{\left\| #1 \right\|}
\newcommand{\snorm}[1]{\| #1 \|}
\newcommand{\set}[1]{\left\{ #1 \right\}}
\newcommand{\abs}[1]{\left| #1 \right|}
\newcommand{\as}{\text{-a.s.}}
\newcommand{\aevery}{\text{-a.e.}}
\newcommand\newAF{\widetilde{A}}

\newcommand{\itref}[1]{\ref{#1}}

\title{\Large \bf
On purely discontinuous additive functionals of subordinate Brownian motions
}

\author{{\bf Zoran Vondra\v{c}ek}\thanks{Research supported in part by the Croatian Science Foundation under the project 3526}
\quad and
\quad {\bf Vanja Wagner$^\ast$}
}

\date{}

\maketitle

\begin{abstract}
Let $A_t=\sum_{s\le t} F(X_{s-},X_s)$ be a purely discontinuous additive functional of a subordinate Brownian motion $X=(X_t, \P_x)$. 
We give a sufficient condition on the non-negative function $F$ that guarantees that finiteness of $A_{\infty}$ implies finiteness of its expectation. This result is then applied to study the relative entropy of  $\P_x$ and the probability measure induced by a purely discontinuous Girsanov transform of the process $X$. We prove these results under the weak global scaling condition on the Laplace exponent of the underlying subordinator. 
\end{abstract}




\section{Introduction}\label{s:intro}
Let $X=(X_t, \P_x)$ be an isotropic L\'evy process in $\R^d$, $d\geq 1$. For a non-negative, bounded and symmetric function $F:\R^d\times \R^d\to [0,\infty)$, define the purely discontinuous additive functional $A_t:=\sum_{0<s\le t}F(X_{s-}, X_s)$. It is often important to understand when does the finiteness of the additive functional at infinity, $A_{\infty}<\infty$, imply the finiteness of the expectation of $A_{\infty}$. To be more precise, we will be interested in sufficient conditions on the function $F$ such that $\P_x(A_{\infty}<\infty)$ for all $x\in \R^d$ implies that $\E_x A_{\infty}<\infty$ for all $x\in \R^d$. In the case of an isotropic $\alpha$-stable process $X$ this question was investigated
in \cite{SV}, cf.~Theorem 4.15. The result of that theorem was further used to study the relative entropy of  $\P_x$ and the probability measure induced by a purely discontinuous Girsanov transform of the stable process $X$, see \cite[Theorem 1.2]{SV}.

The goal of this paper is to extend the results of \cite{SV} from the stable process to a rather large class of subordinate Brownian motions. Instead of the strict scaling property enjoyed by the stable process, we will impose as a substitute the weak global scaling condition. More precisely, let $W=(W_t, \P_x)$ be a $d$-dimensional Brownian motion running twice as fast as the standard Brownian motion, $d\ge 1$, 
and let $S=(S_t)_{t\ge 0}$ be an independent subordinator with the Laplace exponent $\phi$. The process $X=(X_t, \P_x)$ defined by $X_t:=W_{S_t}$ is called a subordinate Brownian motion. It is an isotropic L\'evy process with the characteristic exponent $\psi(\xi)=\phi(|\xi|^2)$. Since any L\'evy process  is completely characterized by its characteristic exponent, we will, without loss of generality, throughout the paper assume that the subordinate Brownian motion $X=(\Omega, \MM, \MM_t, \theta_t, X_t, \P_x)$ is defined on the Skorokhod path space $\Omega=D([0,\infty), \R^d)$ of cadlag functions $\omega:[0,\infty)\to \R^d$ with $X_t(\omega)=\omega(t)$ being projections, $\MM=\sigma\left(\cup_{t\ge 0}\MM_t\right)$, and the shift $(\theta_t \omega)(s)=\omega(s+t)$.  

Our main assumption on the Laplace exponent $\phi$ is the following global scaling condition: There exist constants  $0<\delta_1\leq \delta_2<1\wedge \frac d 2$ and $a_1,a_2>0$  such that
\begin{equation}\label{e:H}
   a_1\lambda^{\delta_1}\leq\frac{\phi(\lambda x)}{\phi(x)}\leq a_2\lambda^{\delta_2},\quad \lambda\geq 1, x>0. 
\end{equation}

Recall that the Laplace exponent $\phi$ is a Bernstein function satisfying $\phi(0)=0$, which implies the representation
$$
\phi(\lambda)=b\lambda +\int_{(0,\infty)}(1-e^{-\lambda t})\, \mu(dt)\, , \quad \lambda >0\, ,
$$
where $\mu$ is the L\'evy measure of $\phi$. Under \eqref{e:H} we have $b=0$ and $\phi(s)\ge a_2^{-1} s^{\delta_2}$ for $s\in (0,1)$. The latter implies that $X$ is transient, cf.~\eqref{e:cf-criterion-2}. 
The function $\phi$ is called a complete Bernstein function if $\mu(dt)=\mu(t)dt$ with a completely monotone density $\mu(t)$, cf.~\cite{SSV}.
For simplicity, let $\Phi(s):=\phi(s^{-2})^{-1}$.
 
Let $F:\R^d\times \R^d\to [0,\infty)$ be symmetric and bounded. The next condition on $F$ will be crucial for our results: Assume that there are constants $C>0$ and $\beta >1$ such that
\begin{equation}\label{e:fuchsian}
     F(x,y)\le C \frac{\Phi(|x-y|)^{\beta}}{1+\Phi(|x|)^{\beta}+\Phi(|y|)^{\beta}}\, ,
     \qquad\text{for all\ \ } x,y\in\R^d.
 \end{equation}
In the context of elliptic diffusion, the analog of \eqref{e:fuchsian} is sometimes called the Fuchsian condition, cf.~\cite{B-AP}.
 The main result of this paper is the following theorem.
 \begin{thm}\label{t:finite-expectation}
Suppose that $X$ is the subordinate Brownian motion via the subordinator whose Laplace exponent is a complete Bernstein function and satisfies \eqref{e:H}. Let $\beta>1$ and assume that $F:\R^d\times\R^d \to [0,\infty)$ is bounded, symmetric and satisfies condition \eqref{e:fuchsian}. Let $A^F_t=\sum_{0<s\le t} F(X_{s-}, X_s)$. If $\P_x(A^F_{\infty}<\infty)=1$ for all $x\in \R^d$, then $\sup_{x\in \R^d} \E_x [A^F_{\infty}]<\infty$.
\end{thm}

It is easy to see, cf.~\cite[Remark 4.16]{SV} in the stable case, that there exists $F$ satisfying \eqref{e:fuchsian} such that $\E_x[A_{\infty}]=\infty$. Of course, in this case it cannot hold that $\P_x(A_{\infty}<\infty)=1$. On the other hand, condition \eqref{e:fuchsian} is almost necessary for the validity of Theorem \ref{t:finite-expectation}. We first need the following definition.

\begin{defn}\label{d:iCb}
    Let $C>0$ and let $\tilde F:[0,\infty)\to[0,1]$ be a bounded non-negative function such that $\tilde F(0)=0$. A bounded symmetric function $F:\R^d\times \R^d\to\R$ vanishing on the diagonal is in the class $I(C,\tilde F)$ if
    \begin{equation}\label{eq:iCb}
        |F(x,y)|\le C \tilde F(|x-y|) \quad\text{for all\ } x,y\in \R^d.
    \end{equation}
\end{defn}

\begin{thm}\label{t:infinite-expectation}
Suppose that $X$ is the subordinate Brownian motion via the subordinator whose Laplace exponent is a complete Bernstein function and satisfies \eqref{e:H}. 
For every $\beta>1$ there exist a non-negative bounded function $F\in I(1,\Phi(\cdot)^{\beta}\wedge 1)$ such that $A^F_{\infty}:=\sum_{s>0} F(X_{s-}, X_s)<\infty$ $\P_x$-a.s., but $\E_x [A^F_{\infty}]=\infty$, for all $x\in \R^d$.
\end{thm}

Proofs of Theorems \ref{t:finite-expectation} and \ref{t:infinite-expectation} follow the roadmap traced in \cite{SV} in the case of the stable process. Of course, due to the fact that neither the jumping kernel nor the Green function are known explicitly in the current situation, and the lack of exact scaling, the proofs are technically more involved and challenging. The main new technical contribution is the proof of Lemma \ref{l:key-lemma}. This lemma is used to prove a Harnack inequality for $F$-harmonic functions with a uniform constant in the class of scaled processes. These results are shown in Section \ref{s:HI} under weak scaling at infinity only. Proofs of Theorems \ref{t:finite-expectation} and \ref{t:infinite-expectation} are given in Sections \ref{s:proof-finite} and \ref{s:proof-infinite} respectively.

As an application of Theorem \ref{t:finite-expectation} we study the question of finiteness of relative entropy of probability measure $\P_x$ and the probability measure induced by a purely discontinuous Girsanov transform of the process $X$. Recall that the subordinate Brownian motion $X=(\Omega, \MM, \MM_t, \theta_t, X_t, \P_x)$ is defined on the Skorokhod path space $\Omega=D([0,\infty), \R^d)$.  

\begin{defn}[\cite{CS03,S}]\label{d:i2x}
    \begin{itemize}
    \item[(a)]\label{d:i2x-a}
    The class $J(X)$ consists of all bounded, symmetric functions $F:\R^d\times\R^d\to\R$ which vanish on the diagonal and satisfy
    $$
        \lim_{t\to 0} \sup_{x\in \R^d} \E_x\left[\int_0^t\int_{\R^d} |F(X_{s-},y)| j(|X_{s-}-y|)\, dy\,  ds\right] = 0.
    $$
    \item[(b)]\label{d:i2x-b}
    The class $I_2(X)$ consists of all bounded, symmetric functions $F:\R^d\times\R^d\to\R$ which vanish on the diagonal and satisfy for all $x\in\R^d$ and $t>0$
    $$
        \E_x\bigg[\sum_{s\le t}F^2(X_{s-},X_s)\bigg]
        =\E_x \left[\int_0^t\!\!\! \int_{\R^d} F^2(X_{s-},y) j(|X_{s-}-y|)\, dy\,  ds\right] <\infty.
    $$
    \end{itemize}
\end{defn}
It is straightforward to see that $J(X)\subset I_2(X)$. It is shown in \cite{S} that for $F\in I_2(X)$ satisfying $\inf_{x,y\in \R^d} F(x,y)>-1$ one can construct a martingale additive functional $M^F=(M^F_t)_{t\ge 0}$ such that the quadratic variation of $M$ is given by 
$$
[M^F]_t=\sum_{0<s\le t}F^2(X_{s-}, X_s)\, .
$$
The Dolean-Dade exponential of $M^F$ defined by $L_t^F:=\EE(M_t^F)$ is under each $\P_x$ a non-negative local martingale, hence a supermartingale. We refer the reader to \cite[Section 2]{S} for details. It is proved in \cite[p.474]{Sh} that there exists a family $(\wt{\P}_x)_{x\in \R^d}$ of (sub)-probability measures on $\MM$ such that
$$
{d\wt{\P}_x}_{|\MM_t}=L_t^F {d\P_x}_{|\MM_t} \quad \text{for all }t\ge 0\, .
$$
Under these measures $X$ is a strong Markov process. We will write $\wt{X}=(\wt{X}_t, \MM. \MM_t, \wt{\P}_x)$ to denote this process and call it the purely discontinuous Girsanov transform of $X$. Since $L_t^F>0$, the measures $\P_x$ and $\wt{\P}_x$ are absolutely continuous on each $\MM_t$. The question when these measures are absolutely continuous on the whole time interval $[0,\infty)$ was answered in \cite[Theorem 1.1]{SV} (for general conservative symmetric right Markov process in $\R^d$). In particular, the relative entropy of measures $\P_x$ and $\wt{\P}_x$, $\HH(\P_x; \wt{\P}_x)<\infty$ if and only if $\E_x\left[\sum_{t>0}F^2(X_{s-}, X_s)\right]<\infty$. 
Recall that the relative entropy of two measures $\nu$ and $\mu$ is defined as
$$
    \HH(\nu; \mu)
    =
    \begin{cases}\displaystyle
        \int \frac{d\nu}{d\mu}\log \frac{d\nu}{d\mu}\, d\mu=\int\log \frac{d\nu}{d\mu} \, d\nu &\text{if\ \ } \nu\ll\mu,\\
         +\infty & \text{otherwise.}
    \end{cases}
$$
By using Theorems \ref{t:finite-expectation} and \ref{t:infinite-expectation} we can prove the following analogs of \cite[Theorems 1.2 and 1.3]{SV}.

\begin{thm}\label{t:entropy}
Suppose that $X$ is the subordinate Brownian motion 
via the subordinator whose Laplace exponent is a complete Bernstein function and satisfies \eqref{e:H}. 
\begin{itemize}
	\item[(a)] Let $F\in I_2(X)$ and $\inf_{x,y\in \R^d}F(x,y)>-1$. Then either $\wt{\P}_x\perp \P_x$ or $\wt{\P}_x\sim \P_x$. If $\wt{\P}_x\sim \P_x$, and if there exist $C>0$ and $\beta>1/2$ such that
	\begin{equation}\label{e:fuchsian-squared-a}
	0\le F(x,y)\le C \frac{\Phi(|x-y|)^{\beta}}{1+\Phi(|x|)^{\beta}+\Phi(|y|)^{\beta}}\, , \quad \text{for all }x,y\in \R^d\, ,
	\end{equation}
	then $\HH(\P_x; \wt{\P}_x)<\infty$.
	\item[(b)] For each $\gamma$ and $\beta$ satisfying $0<\gamma <1/2<\beta$ there exists $F\in I_2(X)$ such satisfying
	\begin{equation}\label{e:fuchsian-squared-b}
	F(x,y)\le \frac12 \frac{\Phi(|x-y|)^{\beta}}{1+\Phi(|x|)^{\gamma}+\Phi(|y|)^{\gamma}}\, , \quad \text{for all }x,y\in \R^d\, ,
	\end{equation}
	such that $\P_x\ll \widetilde{\P}_x$ and $\HH(\P_x; \widetilde{\P}_x)=\infty$.
\end{itemize}
\end{thm}
The proof of Theorem \ref{t:entropy} is given in Section \ref{s:proof-entropy}.

We end this introduction by mentioning another generalization of \cite{SV} that came to our attention after finishing this paper. In his master thesis \cite{Nab}, the author studies the situation when the additive functional is a sum of a continuous  and a purely discontinuous additive functional, and generalizes \cite[Theorem 1.1]{SV} to this setting. Under two additional conditions, \emph{unavoidable set condition} and \emph{uniform Harnack inequality}, he also generalizes \cite[Theorem 1.2]{SV}. These results are disjoint from the ones presented in this paper.

\section{Harnack inequality for $F$-harmonic functions}\label{s:HI}
As in the introduction, let $W=(W_t, \P_x)$ be the standard $d$-dimensional Brownian motion and  $S=(S_t)$ an independent subordinator with the Laplace exponent $\phi$, no drift and the infinite L\'evy measure $\mu(dt)$. 
Without loss of generality we assume that $\phi(1)=1$. The subordinate Brownian motion $X$ is defined by $X_t:=W_{S_t}$. It is an isotropic L\'evy process with the characteristic exponent $\psi(\xi)=\phi(|\xi|^2)$, $\xi\in \R^d$, and the L\'evy measure $\nu(dx)=j(x)dx$ where the radial density $j(x)=j(|x|)$ is given by 
$$
j(r)=\int_0^{\infty}(4\pi t)^{-d/2}e^{-\frac{r^2}{4t}}\mu(dt)\, .
$$
The process $X$ has continuous transition densities given by
$$
p(t,x,y)=\int_{(0,\infty)}(4\pi s)^{-d/2}e^{-\frac{|x-y|^2}{4s}}\, \P(S_t\in ds),
$$
and is therefore strongly Feller.

We recall that for every Bernstein function $\phi$ it holds that
\begin{equation}\label{e:phi}
  (1\wedge\lambda)\phi(x)\leq \phi(\lambda x)\leq(1\vee \lambda)\phi(x),\quad \lambda >0,\,x>0.
\end{equation}
In this section we assume that $\phi$ satisfies the following weak scaling condition at infinity: There exist constants  $0<\delta_1\leq \delta_2<1\wedge  \frac{d}{2}$ and $a_1,a_2>0$  such that
\begin{equation}\label{e:H-infty}
   a_1\lambda^{\delta_1}\leq\frac{\phi(\lambda x)}{\phi(x)}\leq a_2\lambda^{\delta_2},\quad \lambda\geq 1, x>1. 
\end{equation}
Note that this condition is weaker than the one in \eqref{e:H}.  Furthermore, we assume that $X$ is transient.
According to the Chung-Fuchs type criterion this is equivalent to
\begin{equation}\label{e:cf-criterion}
\int_0^1 \frac{s^{\frac{d}{2}-1}}{\phi(s)}\, ds <\infty\, ,
\end{equation}
and imposes an additional assumption only in cases $d=1$ and $d=2$.

Recall that $\Phi(s)=\phi(s^{-2})^{-1}$. When $\phi$ satisfies \eqref{e:H-infty}, then 
\begin{equation}\label{e:lower-bound-Phi}
\Phi(s)\le a_1^{-1}s^{2\delta_1}\, , \quad \text{for all }s\in (0,1]\, .
\end{equation}

Let $B_r=B(0,r)$ be the ball of radius $r>0$ centered at the origin, $\delta_{B_r}(x)$ the distance of $x$ to the boundary $\partial B_r$,  and let $G_{B_r}$ denote the Green function of the process $X$ killed upon exiting $B_r$. Let further,
$$
P_{B_r}(x,z)=\int_{B_r}G_{B_r}(x,y)j(|y-z|)\, dy\, , \qquad x\in B_r, z\in \overline{B_r}^c,
$$
be the Poisson kernel of the ball $B_r$. We will need the following three estimates of the L\'evy density $j(r)$, Green function $G_{B_r}$ and the Poisson kernel $P_{B_r}$: If $\phi$ is a complete Bernstein function satisfying \eqref{e:H-infty}, then there exist constants $C_1\ge 1$, $C_2\ge 1$ and $C_3>0$, depending only on dimension $d$ and the constants $a_1,a_2, \delta_1, \delta_2$ from \eqref{e:H-infty}, such that for all $r\in (0,1]$,
\begin{equation}
C_1^{-1} r^{-d}\Phi(r)^{-1}\le j(r) \le C_1 r^{-d}\Phi(r)^{-1}\, , \label{e:estimate-of-j}
\end{equation}
\begin{eqnarray}
\lefteqn{ C_2^{-1} \frac{\Phi(|x-y|)}{|x-y|^d}\left(1\wedge\frac{\Phi(\delta_{B_r}(x))^\frac 1 2\Phi(\delta_{B_r}(y))^\frac 1 2}{\Phi(|x-y|)}\right) \le G_{B_r}(x,y)}\label{e:estimate-of-G}\\
&\le & C_2 \frac{\Phi(|x-y|)}{|x-y|^d}\left(1\wedge\frac{\Phi(\delta_{B_r}(x))^\frac 1 2\Phi(\delta_{B_r}(y))^\frac 1 2}{\Phi(|x-y|)}\right), \quad x,y\in B_r, \nonumber 
\end{eqnarray}
\begin{equation}
P_r(0,z) \ge C_3 j(|z|) \Phi(r)\, , \qquad z\in \overline{B_r}^c\, \label{e:estimate-of-P}
\end{equation}
See \cite[Lemma 3.2]{KSV14} for \eqref{e:estimate-of-j}, \cite[Theorem 1.2]{KM} for \eqref{e:estimate-of-G}  (cf.~also \cite[Proposition 7.5]{KSV16}), and \cite[Lemma 2.6]{KSV15} for \eqref{e:estimate-of-P}. Note that in those results it is stated that the constant depends on $\phi$, but a closer inspection of the proof reveals that the dependence on $\phi$ is only through the constants in \eqref{e:H-infty}.
As a consequence of \eqref{e:estimate-of-G} we have the following estimates.

\begin{lemma}\label{l:green}
There exists a constant $C_4=C_4(a_1, a_2, \delta_1, \delta_2,d)$ such that for every $r\in (0, 1]$,
    \begin{gather}\label{e:3G}
        \frac{G_{B_r}(x,y)G_{B_r}(z,w)}{G_{B_r}(x,w)}
        \leq
        \begin{cases}
            \dfrac{C_4\,\Phi(|x-y|)\Phi(|z-w|)}{\Phi(|x-w|)}\dfrac{|x-w|^{d}}{|x-y|^{d}|z-w|^{d}},         & (x,w)\in E_r,\\[\bigskipamount]
            C_4\,\Phi(|x-y|)^\frac 1 2\Phi(|z-w|)^\frac 1 2\dfrac{|x-w|^{d}}{|x-y|^{d}|z-w|^{d}},   & (x,w)\not\in E_r,
        \end{cases}
    \end{gather}
    where $E_r=\{(x,w)\in B_r\times B_r:|x-w|\leq\frac12 \max\{\delta_{B_r}(x), \delta_{B_r}(w)\}\}$.
\end{lemma}
\begin{proof} 
We use the Green function estimates \eqref{e:estimate-of-G}. First, let $(x,w)\in E_r$ and without loss of generality let $\delta_{B_r}(w)\leq \delta_{B_r}(x)$. From \eqref{e:estimate-of-G} we get
    \begin{align*}
     \frac{G_{B_r}(x,y)G_{B_r}(z,w)}{G_{B_r}(x,w)}\leq c_1
     \frac{\Phi(|x-y|)\Phi(|z-w|)}{\Phi(|x-w|)\wedge(\Phi(\delta_{B_r}(x))^\frac 1 2\Phi(\delta_{B_r}(w))^\frac 1 2)}\frac{|x-w|^d}{|x-y|^d|z-w|^d}.
    \end{align*} 
    Since $\frac 1 2\delta_{B_r}(x)\geq|x-w|\geq \delta_{B_r}(x)-\delta_{B_r}(w)$, it follows that $\delta_{B_r}(w)\geq \frac 1 2 \delta_{B_r}(x)$ and therefore,
    \[
     \Phi(|x-y|)\leq \Phi(\tfrac 1 2 \delta_{B_r}(x))\leq \Phi(\delta_{B_r}(x))^\frac 1 2\Phi(\delta_{B_r}(w))^\frac 1 2.
    \]
    This implies \eqref{e:3G} in the first case. On the other hand, if $(x,w)\not\in E_r$ then
     \[
     \Phi(|x-y|) \geq \Phi(\tfrac 1 2\delta_{B_r}(x))^\frac 1 2\Phi(\tfrac 1 2\delta_{B_r}(w))^\frac 1 2\overset{\eqref{e:phi}}{\geq}\tfrac 1 4\Phi(\delta_{B_r}(x))^\frac 1 2\Phi(\delta_{B_r}(w))^\frac 1 2.
    \]
    Therefore, by \eqref{e:estimate-of-G}
    \[
   G_{B_r}(x,w)\geq c_2\frac{\Phi(\delta_{B_r}(x))^{\frac 1 2}\Phi(\delta_{B_r}(w))^{\frac 1 2}}{|x-w|^d}.
  \]
  Furthermore, since $\delta_{B_r}(y)\leq 2(\delta_{B_r}(x)\vee|x-y|)$, \eqref{e:estimate-of-G} implies that
  \[
   G_{B_r}(x,y)\leq c_3\frac{\Phi(|x-y|)^{\frac 1 2}\Phi(\delta_{B_r}(x))^{\frac 1 2}}{|x-y|^d}.
  \]
  and analogously, 
  \[
   G_{B_r}(z,w)\leq c_3\frac{\Phi(|z-w|)^{\frac 1 2}\Phi(\delta_{B_r}(w))^{\frac 1 2}}{|z-w|^d}.
  \]
  Using these three estimates of the Green function, we get \eqref{e:3G} in the second case. \end{proof}

\begin{lemma}\label{l:kato-class}
Let $\tilde{F}:[0,\infty)\to [0,1]$ be a non-negative bounded function such that $\tilde{F}(0)=0$ and 
\begin{equation}\label{e:tilde-F-Phi}
\int_0^1 \frac{\tilde{F}(s)}{\Phi(s)s}\, ds <\infty\, .
\end{equation}

\noindent
(a) If $\tilde{A}_t:=\sum_{0<s\le t}\tilde{F}(|X_s-X_{s-}|)$, then $\E_x \tilde{A}_t=ct$ where $c=\int_{\R^d} \tilde{F}(|y|)j(|y|)dy$. Consequently, 
$\lim_{t\to 0} \sup_{x\in \R^d} \E_x \tilde{A}_t =0$. 

\noindent
(b) Let $C>0$, suppose that  $F\in I(C,\tilde F)$ and define $A_t^F:=\sum_{0<s\le t} F(X_{s-},X_s)$.  Then $\lim_{t\to 0} \sup_{x\in \R^d} \E_x |A_t^F| =0$, i.e., $F\in J(X)$.
\end{lemma}
\begin{proof} (a) By the the L\'evy system formula we have
    \begin{align*}
        \E_x [\tilde{A}_t]
        = \E_x \left[\sum_{s\le t} \tilde{F}(|X_{s}-X_{s-}|) \right]
        &= \E_x \left[\int_0^t\!\!\! \int_{\R^d} \tilde{F}(|y-X_{s-}|) j(|y-X_{s-}|)\, dy \, ds \right]\\
        &= \E_x \left[\int_0^t h(X_s)\, ds\right]
    \end{align*}
    where
    \begin{align*}
        h(y)
        &:=\int_{\R^d} \tilde{F}(|y-z|)j(|y-z|)\, dz =\int_{\R^d} \tilde{F}(|z|)j(|z|)\, dz \\
        &\leq c_1\int_{|z|\le 1} \dfrac{\tilde F(|z|)}{\Phi(|z|)|z|^d}\,dz +\int_{|z|>1} j(|z|)\, dz = c_2 <\infty.
    \end{align*}
    Here the second line follows from \eqref{e:estimate-of-j} and the assumption \eqref{e:tilde-F-Phi}.
    Therefore, $\E_x [\tilde{A}_t] =ct $ with $c:=\int_{\R^d} \tilde{F}(|y|)j(|y|)dy$.

\noindent (b) This follows directly from (a).   \end{proof}

By using \eqref{e:lower-bound-Phi}, it is straightforward to check that for any $\beta>1$, the function $\tilde{F}(s):=\Phi(s)^{\beta}\wedge 1$ satisfies the assumptions of the previous lemma.

\begin{lemma}\label{l:key-lemma}
    Let $\beta>1$, $C>0$ and $\tilde F(s):=\Phi(s)^\beta \wedge 1$. For every $\varepsilon>0$ there exists a constant $r_0=r_0(d,a_1,a_2,\delta_1,\delta_2,\beta,C,\varepsilon)\in (0,1]$ such that for every $r\in (0,r_0]$ and $F\in I(C,\tilde F)$ 
    $$
        \sup_{x,w\in B_r} \int_{B_r}\int_{B_r} \frac{G_{B_r}(x,y)G_{B_r}(z,w)}{G_{B_r}(x,w)}\, |F(y,z)| j(|y-z|)\, dz\, dy <\varepsilon.
    $$
\end{lemma}
\begin{proof}    
    The ratio of Green functions in the integral above is by Lemma \ref{l:green} less than or equal to the sum of two expressions on the right-hand side of \eqref{e:3G}.  Hence it suffices to separately estimate the integral when the ratio is replaced by either of the two expressions.
    
    Let $0<s\leq t\leq 1$. First note that it follows from \eqref{e:phi} that for $d\geq 2$,
    \begin{equation}\label{e:phi-rd-incr}
     \phi(s^{-2})s^d\leq\frac{t^2}{s^2}\phi(t^{-2})s^d\leq \phi(t^{-2})t^d,
    \end{equation}
and in case of $d=1$ it follows from $\eqref{e:H-infty}$ that 
    \begin{equation}\label{e:phi-rd-incr-d1}
    s\phi(s^{-2})\le a_2 s\left(\frac{t}{s}\right)^{2\delta_2}\phi(t^{-2})\le a_2 s\left(\frac{t}{s}\right)\phi(t^{-2}) =a_2 t\phi(t^{-2})\, .
    \end{equation}
Furthermore, for $d\geq 1$ and all $0<s\leq t\leq 1$,
	\begin{equation}\label{e:phi-rd-triangle}
      \phi((s+t)^{-2})(s+t)^d\leq 2^d \phi(t^{-2})t^d\leq 2^d(\phi(s^{-2})s^d+\phi(t^{-2})t^d)\, .
	\end{equation}
   By using first \eqref{e:phi-rd-incr} (respectively \eqref{e:phi-rd-incr-d1} in case $d=1$), and then \eqref{e:phi-rd-triangle} we get
    \begin{align*}
     &\int_{B_r}\int_{B_r}\dfrac{\Phi(|x-y|)\Phi(|z-w|)}{\Phi(|x-w|)}\dfrac{|x-w|^{d}}{|x-y|^{d}|z-w|^{d}}|F(y,z)|j(|y-z|)\,dy\, dz\\
    &\leq c_1\int_{B_r}\int_{B_r}\dfrac{\Phi(|z-w|)}{|z-w|^{d}}|F(y,z)|j(|y-z|)\,dy\, dz\\ 
     &+ c_1\int_{B_r}\int_{B_r}\dfrac{\Phi(|x-y|)\Phi(|z-w|)}{\Phi(|y-z|)}\dfrac{|y-z|^{d}}{|x-y|^{d}|z-w|^{d}}|F(y,z)|j(|y-z|)\,dy\, dz\\
     &+ c_1\int_{B_r}\int_{B_r}\dfrac{\Phi(|x-y|)}{|x-y|^{d}}|F(y,z)|j(|y-z|)\,dy\, dz=:c_1(I_1+I_2+I_3),
    \end{align*}
    for some $c_1=c_1(d)$. By \eqref{e:estimate-of-j}, for $x,w\in B_r$
    \begin{align*}
     I_1&\leq c_2\int_{B(w,2r)}\dfrac{\Phi(|z-w|)}{|z-w|^{d}}\left(\int_{B(z,2r)}\dfrac{|F(y,z)|}{\Phi(|y-z|)|y-z|^d}\,dy\right)\, dz\\
     &\leq c_3 \int_{0}^{2r} \dfrac{\Phi(s)}{s}ds\cdot \int_{0}^{2r} \dfrac{\Phi(s)^\beta}{\Phi(s)s}ds\\
     &\overset{\eqref{e:H-infty}}{\leq}
     \frac{c_3}{a_1^\beta}(2r)^{-2\delta_1\beta}\Phi(2r)^\beta\int_0^{2r} s^{2\delta_1-1}\,ds\cdot\int_0^{2r} s^{2\delta_1(\beta-1)-1}\,ds\\     
     &\overset{\eqref{e:phi}}{\leq}\frac{c_34^{\beta-1}}{a_1^\beta\delta_1^2(\beta-1)} \Phi(r)^\beta.
    \end{align*}
    Note that the same upper bound holds for $I_3$. For the integral $I_2$ it follows by \eqref{e:estimate-of-j} that
    \begin{align*}
     I_2&\leq c_4\int_{B_r}\dfrac{\Phi(|x-y|)}{|x-y|^{d}}\left(\int_{B_r}\dfrac{\Phi(|z-w|)}{|z-w|^{d}}\Phi(|y-z|)^{\beta-2}\,dz\right)\,dy.
     \end{align*}
Without  loss of generality, assume $\beta\in(1,2)$. It follows that that     \begin{align*}
     I_2&\leq c_4\int_{B_r}\dfrac{\Phi(|x-y|)}{|x-y|^{d}}\left(\int\limits_{\substack{|z|<r\\|z-w|<|y-z|}}\dfrac{\Phi(|z-w|)}{|z-w|^{d}}\Phi(|y-z|)^{\beta-2}\,dz\right)\,dy\\
     &+c_4\int_{B_r}\dfrac{\Phi(|x-y|)}{|x-y|^{d}}\left(\int\limits_{\substack{|z|<r\\|z-w|\geq|y-z|}}\dfrac{\Phi(|z-w|)}{|z-w|^{d}}\Phi(|y-z|)^{\beta-2}\,dz\right)\,dy\\
     &\leq c_6\int_{B(x,2r)}\dfrac{\Phi(|x-y|)}{|x-y|^{d}}\left(\int\limits_{\substack{B(w,2r)}}\dfrac{\Phi(|z-w|)^{\beta-1}}{|z-w|^{d}}\,dz\right)\,dy\\
     &+c_6\int_{B(x,2r)}\dfrac{\Phi(|x-y|)}{|x-y|^{d}}\left(\int\limits_{\substack{B(y,2r)}}\dfrac{\Phi(|z-y|)^{\beta-1}}{|z-y|^{d}}\,dz\right)\,dy\\     
     &\leq c_7 \int_{0}^{2r} \dfrac{\Phi(s)}{s}ds\int_{0}^{2r} \dfrac{\Phi(s)^{\beta-1}}{s}ds\overset{\eqref{e:H-infty}}{\leq} \frac{c_74^{\beta-1}}{a_1^\beta\delta_1^2(\beta-1)}\Phi(r)^{\beta}.
     \end{align*}
Furthermore, by using the inequality $(s+t)^d\le 2^d(s^d+t^d)$, we have
    \begin{align*}
     &\int_{B_r}\int_{B_r}\Phi(|x-y|)^\frac 1 2\Phi(|z-w|)^\frac 1 2\dfrac{|x-w|^{d}}{|x-y|^{d}|z-w|^{d}}|F(y,z)|j(|y-z|)\,dy\, dz\\
    &\leq c_{8}\int_{B_r}\int_{B_r}\Phi(|x-y|)^\frac 1 2\dfrac{\Phi(|z-w|)^\frac 1 2}{|z-w|^{d}}|F(y,z)|j(|y-z|)\,dy\, dz\\ 
     &+c_{8}\int_{B_r}\int_{B_r}\Phi(|x-y|)^\frac 1 2\Phi(|z-w|)^\frac 1 2\dfrac{|y-z|^{d}}{|x-y|^{d}|z-w|^{d}}|F(y,z)|j(|y-z|)\,dy\, dz\\
     &+c_{8}\int_{B_r}\int_{B_r}\Phi(|z-w|)^\frac 1 2\dfrac{\Phi(|x-y|)^\frac 1 2}{|x-y|^{d}}|F(y,z)|j(|y-z|)\,dy\, dz=:c_{8}(I_4+I_5+I_6).
    \end{align*}
Similarly as in the case of integral $I_1$ it follows that
    \begin{align*}
     I_4&\leq c_{9}\Phi(2r)^\frac 1 2\int_{B(w,2r)}\dfrac{\Phi(|z-w|)^\frac 1 2}{|z-w|^{d}}\left(\int_{B(z,2r)}\dfrac{\Phi(|y-z|)^\beta}{\Phi(|y-z|)|y-z|^d}\,dy\right)\, dz\\
     &\leq c_{10} \Phi(r)^{\frac12}\int_{0}^{2r} \dfrac{\Phi(s)^\frac 1 2}{s}\,ds\cdot \int_{0}^{2r} \dfrac{\Phi(s)^{\beta-1}}{s}\,ds\overset{\eqref{e:H-infty}}{\leq} c_{11}\Phi(r)^\beta
    \end{align*}
    and analogously $I_6\leq c_{11}\Phi(r)^\beta$. Finally, since $\Phi$ is increasing and $\beta>1$,  
    \begin{align*}
     I_5&\leq c_{9}\int_{B_r}\dfrac{\Phi(|x-y|)^\frac 1 2}{|x-y|^{d}}\left(\int_{B_r}\dfrac{\Phi(|z-w|)^\frac 1 2}{|z-w|^{d}}\Phi(|y-z|)^{\beta-1}\,dz\right)\,dy\\
     &\leq c_{12}\Phi(r)^{\beta-1}\left(\int_0^{2r}\dfrac{\Phi(s)^\frac 1 2}{s}\, ds\right)^2\overset{\eqref{e:H-infty}}{\leq} c_{13} \Phi(r)^\beta.
    \end{align*}
Combining the above inequalities, and by using the estimates from Lemma \ref{l:green}, we get that for $r\leq1$,  
\begin{align}\label{e:key-lemma1}
\sup\limits_{x,w\in B_r} &\int_{B_r}\int_{B_r} \frac{G_{B_r}(x,y)G_{B_r}(z,w)}{G_{B_r}(x,w)}\, |F(y,z)| j(|y-z|)\, dz\, dy\leq c_{14}\Phi(r)^{\beta}\overset{\eqref{e:H-infty}}{\leq} \frac{c_{14}}{a_1}r^{2\delta_1\beta}\, ,
\end{align}
where the constant $c_{14}$ depends only on $d, a_1,a_2, \delta_1,\delta_2, C$ and $\beta$. By choosing $r_0$ such that $c_{14}a_1^{-1}r_0^{2\delta_1\beta}< \varepsilon$, we finish the proof.  
\end{proof}

Let $F:\R^d\times \R^d\to [0,\infty)$ be bounded and symmetric and set $A^{F}_{t}:=\sum_{0<s\le t}F(X_{s-}, X_s)$. For $x,w\in B(x_0,r)$, let $\P_x^{w}$ denote the law of the $h$-transformed killed process $X^{B(x_0,r)}$ with respect to the excessive function $G_{B(x_0,r)}(\cdot, w)$, i.e.~the process $X^{B(x_0,r)}$ conditioned to die at $\{w\}$. By \cite[Proposition 3.3]{CS03} it holds that
\begin{align}
    \E_x^{w} &\left[\sum_{s\le t}F(X_{s-}^{B(x_0,r)}, X_s^{B(x_0,r)})\right] \nonumber \\
    &=\E_x\left[\int_0^t \int_{B(x_0,r)} \frac{F(X_s^{B(x_0,r)},z)G_{B(x_0,r)}(z,w)}{G_{B(x_0,r)}(x,w)}\, j(|X_s^{B(x_0,r)}-z|)\, dz\, ds\right].  \label{e:levy-system-conditioned-0} 
\end{align}
Let $\zeta=\tau_{B(x_0,r)\setminus \{w\}}$ denote the lifetime of $X$ under the conditional probability $\P_x^w$. Then it follows from \eqref{e:levy-system-conditioned-0} that
\begin{align}
   \E_x^{w} & \left[\sum_{s <\zeta }F(X_{s-}, X_s) \right]  \nonumber  \\
    &= \E_x\left[\int_0^{\tau_{B(x_0,r)}} \int_{B(x_0,r)} \frac{F(X_s,z)G_{B(x_0,r)}(z,w)}{G_{B(x_0,r)}(x,w)}\, j(|X_s-z|)\, dz\, ds\right] \nonumber\\
    &=  \int_{B(x_0,r)} \int_{B(x_0,r)} \frac{G_{B(x_0,r)}(x,y)G_{B(x_0,r)}(z,w)}{G_{B(x_0,r)}(x,w)}\, F(y,z) j(|y-z|)\, dz\, dy.\label{e:levy-system-conditioned}
\end{align}

\begin{lemma}\label{l:estimate-for-killing}
 Let $\beta>1$, and $\tilde F(s):=\Phi(s)^\beta\wedge 1$. Assume that $F$ is non-negative and $F\in I(C,\tilde F)$, let $\varepsilon >0$ and denote by $r_0=r_0(d,a_1,a_2,\delta_1,\delta_2,\beta,C, \varepsilon)>0$ the constant from Lemma \ref{l:key-lemma}. Then for $r\le r_0$ and $\tau=\tau_{B(x_0,r)}$ it holds that
    $$
        e^{-\varepsilon}\le \E_x^{w} \left[e^{-A^{F}_{\tau}}\right] \le 1.
    $$
\end{lemma}
\begin{proof}
    Let $r\le r_0$. By \eqref{e:levy-system-conditioned} and Lemma \ref{l:key-lemma} we see that $\E_x^{w}[A^{F}_{\tau}]<\varepsilon$. By Jensen's inequality, it follows that
    \begin{gather*}
         \E_x^{w}[e^{-A^{F}_{\tau}}]\geq e^{-\E_x^{w}[A^{F}_{\tau}]}\geq e^{-\varepsilon}.
    \end{gather*}
\end{proof}

\begin{defn}
 Let $F$ be a non-negative, symmetric function on $\R^d \times\R^d$ and set $A^F_t:=\sum_{0<s\le t}F(X_{s-},X_s)$. We say that a non-negative function $u:\R^d\to [0,\infty)$ is \emph{$F$-harmonic} in a bounded open set $D$ with respect to $X$ if for every open set $V\subset \overline{V}\subset D$ the following mean-value property holds:
$$
    u(x)=\E_x\left[e^{-A^F_{\tau_{V}}} u(X_{\tau_{V}})\right],\quad \text{for all }x\in V.
$$
The function $u$ is \emph{regular $F$-harmonic in $D$}  if $u(x)=\E_x\left[e^{-A^F_{\tau_{D}}} u(X_{\tau_{D}})\right]$ for all $x\in D$.
\end{defn}
We note that the standard argument using the strong Markov property shows that any regular $F$-harmonic function $u$ is also $F$-harmonic in $D$.

 Next we prove the Harnack inequality for non-negative $F$-harmonic functions.

\begin{thm}\label{t:hi}
    Let $D\subset \R^d$ be a bounded open set and $K\subset D$ a compact subset of $D$. Fix $\beta>1$ and $C>0$, and let $\tilde F(s)=\Phi(s)^{\beta}\wedge 1$. There exists a constant $C_5=C_5(d, a_1,a_2,\delta_1, \delta_2,\beta,C, D, K)>1$ such that for every $F\in I(C,\tilde F)$ and every $u:\R^d\to [0,\infty)$ which is $F$-harmonic with respect to $X$ in $D$, it holds that
    $$
        C_5^{-1}u(x) \le u(y) \le C_5 u(x),\quad x,y \in K.
    $$
\end{thm}
\begin{proof}
    Set $\rho_0=r_0\wedge \tfrac{\text{dist}(K,D^c)}2$ where $r_0=r_0(d,a_1,a_2,\delta_1,\delta_2,\beta,C,\ln(2))$ is the constant from Lemma \ref{l:key-lemma}.
 Let $x\in K$, $r\in (0, \rho_0]$ 
 and $\tau=\tau_{B(x,r)}$. By \cite[Theorem 2.4]{CS03b}, for every $y\in B(x,r)$,
    $$
        u(y)
        =\E_y\left[u(X_{\tau})e^{-A^{F}_{\tau}}\right]
        =\E_y\left[u(X_{\tau})\E_y^{X_{\tau-}} \left[e^{-A^{F}_{\zeta}}\right]\right],
    $$
    where $\zeta=\tau_{B(x, r)\setminus \{v\}}$ 
    and $v=X_{\tau-}$. By Lemma \ref{l:estimate-for-killing}, $\frac12\le \E_x^{X_{\tau-}} \left[e^{-A^{F}_{\zeta}}\right] \le 1$, implying that for every $y\in B(x,r)$
    \begin{equation}\label{e:hi-1}
        \frac12 \E_y[u(X_{\tau})]\le u(y)\le \E_y[u(X_{\tau})].
    \end{equation}
    If $y\in B(x, \tfrac r2)$, then by \cite[Proposition 2.3]{KSV15} for $a=\tfrac 1 2$,
    \begin{align*}
        \E_x [u(X_{\tau})] &= \int_{B(x,r)^c}P_{B(x,r)}(x,z) u(z)\, dz \\
        &\le c_1\, \int_{B(x,r)^c}P_{B(x,r)}(y,z) u(z)\, dz\\
        &=c_1 \E_y [u(X_{\tau})],
    \end{align*}
    where the constant $c_1$, although not explicitly mentioned in the statement of the theorem, depends only on $d, a_1, a_2, \delta_1$ and $\delta_2$. Analogously,
    $$
        \E_x [u(X_{\tau})]\ge c_1^{-1} \E_y [u(X_{\tau})].
    $$
    Combining the last two estimates with \eqref{e:hi-1} yields
    \begin{equation}\label{e:hi-2}
        \tfrac 1 2 c_1^{-1} u(y) \le u(x)\le 2c_1 u(y),\quad y\in B(x,r/2).
    \end{equation}
    In particular, \eqref{e:hi-2} holds for $r=\rho_0$.

    Now pick $z\in K$ such that $z\in B(x,\frac {\rho_0}2)^c$. Let $B_1=B(x,\frac{\rho_0}4)$ and $B_2=B(z,\frac{\rho_0}4)$ and note that $B_1\cap B_2=\emptyset$. It follows that 
    \begin{align*}
        u(z)
        &\overset{\eqref{e:hi-1}}{\ge} \frac12 \E_z[u(X_{\tau_{B_2}})]\\
        &=   \frac12 \int_{B_2^c} u(y)P_{B_2}(z,y)\, dy\\
        &\ge \frac12 \int_{B_1} u(y)P_{B_2}(z,y)\, dy \\
        &\overset{\eqref{e:hi-2}}{\ge}  \frac1{4 c_1}u(x)\int_{B_1} P_{B_2}(z,y)\, dy 
    \end{align*}
   By \eqref{e:estimate-of-P} it follows that 
   \[
    P_{B_2}(z,y)\geq c_2 j(|z-y|)\Phi(\tfrac{\rho_0}{4}),\quad y\in \overline{B(z,\tfrac{\rho_0}{4})}^c, 
   \]
   for some $c_2=c_2(d,a_1,a_2,\delta_1,\delta_2)$. Furthermore, since $j$ is decreasing it follows that
    \begin{align*}
        u(z)
        &\ge \frac{c_2}{4c_1} j\left(\text{diam}(K)\right)\Phi\left(\frac{\rho_0}{4}\right)|B(0,1)|(\rho_0/4)^d u(x)\\
        &\overset{\eqref{e:estimate-of-j}}{\ge} c_3\frac{\Phi\left(\frac{\rho_0}{4}\right)}{\Phi\left(\text{diam}(K)\right)}\left(\frac{\rho_0}{\text{diam}(K)}\right)^{d} u(x)\\
         &\overset{\eqref{e:H-infty}}{\ge} c_4 \left(\frac{\rho_0}{\text{diam}(K)}\right)^{2\delta_1+d} u(x).
    \end{align*}
    Analogously, $u(x)\ge c_4 \left(\frac{\rho_0}{\text{diam}(K)}\right)^{2\delta_1+d} u(z)$. Together with \eqref{e:hi-2} this proves the theorem.
\end{proof}

\begin{remark}{\rm 
It is clear from the proof that the dependence of $C_5$ on $K$ and $D$ is only through the ratio $\rho_0/\mathrm{diam}(K)=(r_0\wedge\frac12 \mathrm{dist}(K,D^c))/\mathrm{diam}(K)$.
}
\end{remark}

\section{Proof of Theorem \ref{t:finite-expectation}}\label{s:proof-finite}
In this section we assume that $\phi$ is a complete Bernstein functions satisfying the global weak scaling condition \eqref{e:H}. 
Note that  \eqref{e:H} implies that 
\begin{equation}\label{e:cf-criterion-2}
 \int_0^1\frac{s^{\frac d 2-1}}{\phi(s)}\,ds\leq a_2\int_0^1 s^{\frac d 2-\delta_2-1}\,ds<\infty\, ,
\end{equation}
so by the Chung-Fuchs type criterion it follows that $X$ is transient.
As in the previous section, we assume that $\phi(1)=1$.

For each $R>0$ define
\[
  \phi^R(s)=\frac{\phi(R^{-2}s)}{\phi(R^{-2})},\quad s>0,
\]
and note that the function $\phi^R$ is also a complete Bernstein function satisfying the scaling condition \eqref{e:H} with same constants. Also, for $s>0$ let $\Phi^R(s):=(\phi^R(s^{-2}))^{-1}$
and $\Phi:=\Phi^1$.  Denote by $X^R$ the subordinate Brownian motion with the characteristic exponent $\psi^R(\xi)=\phi^R(|\xi|^2)$, $\xi\in\R^d$, and note that 
\begin{equation}\label{eq:eqD}
  (X^R_t)_{t\geq 0}\overset{\text{D}}{=} (R^{-1}X_{t/\phi(R^{-2})})_{t\geq 0}.
\end{equation}
Since the Laplace exponent of $X^R$ satisfies \eqref{e:H} with the same constants, the definitions and the results of Section \ref{s:HI} apply to all $X^R$ simultaneously. The notions related to the process $X^R$ will have the superscript $R$. For example, if $F:\R^d\times \R^d\to [0,\infty)$, we let
$$
A^{R, F}_t=\sum_{0<s\le t}F(X_{s-}^R, X_s^R)\, ,
$$
so that $A^{1,F}=A^F$. Similarly, for a Borel set $D\subset \R^d$, we let $\tau_D^R=\inf\{t>0: X_t^R\notin D\}$ be the first exit time from $D$, so that $\tau_D^1=\tau_D$.

Further, for $u:\R^d\to [0,\infty)$, $F:\R^d\times \R^d\to [0,\infty)$, $D\subset \R^d$, and any $R>0$, set
$$
u_R(x):=u(Rx)\, , \quad F_R(x,y):=F(Rx,Ry)\, , \quad D_R:=\{Rx:\, x\in D\}\, .
$$
The following lemma is a crucial result relating regular $F$-harmonic function with respect to $X$ with regular $F_R$-harmonic functions with respect to $X^R$.

 \begin{lemma}\label{l:harmonic-scaling-b}
     Let $D$ be a bounded open set in $\R^d$, $R>0$, $\zeta:=\tau_{D_R}$ and $\eta:=\tau^R_D$. Assume that $u$ is regular $F$-harmonic in $D_R$ for $X$, i.e.
     \begin{equation}\label{e:harmonic-scaling-1}
         u(x)=\E_x\left[e^{-A^F_{\zeta}}u(X_{\zeta})\right]\quad \text{for all\ \ } x\in D_R.
     \end{equation}
     Then $u_R$ is regular $F_R$-harmonic in $D$ for $X^{R}$, i.e.
     \begin{equation*}\label{e:harmonic-scaling-2}
         u_R(x)=\E_x\left[e^{-A^{R,F_R}_{\eta}}u_R(X^{R}_{\eta})\right]\quad \text{for all\ \ } x\in D.
     \end{equation*}
 \end{lemma}
 \begin{proof}
     Note that the $\P_x$-distribution of $(R X^R_t)_{t\ge 0}$ is equal to the $\P_{Rx}$-distribution of $(X_{t\Phi(R)})_{t\ge 0}$. From this identity it follows that the $\P_x$-distribution of the pair $(\eta, R X^R_{\eta})$ is equal to the $\P_{Rx}$-distribution of  $(\zeta/\Phi(R),X_{\zeta})$. Using these scaling identities in the second line and a change of variables in the third line we get
     \begin{align*}
         \E_x\left[e^{-A^{R,F_R}_{\eta}}u_R(X^R_{\eta})\right]
         &=  \E_x\left[e^{-\sum_{s\leq \eta}F(R X^R_{s-}, R X^R_s)}u(RX^R_{\eta})\right]\\
         &=  \E_{Rx}\left[e^{-\sum_{s\leq \zeta/\Phi(R)} F(X_{ \Phi(R)s-},X_{ \Phi(R)s})} u(X_{\zeta})\right]\\
         &=  \E_{Rx} \left[e^{-\sum_{s\leq \zeta}F(X_{s-},X_s)} u(X_{\zeta})\right]\\
         &\overset{\eqref{e:harmonic-scaling-1}}{=}  u(Rx)
         =  u_R(x).
     \end{align*}
     \end{proof}
     
The analog of the following lemma is proved in \cite{SV}, so we omit the proof.
 \begin{lemma}\label{l:harmonic-semilocal}\cite[Lemma 4.10]{SV}
     Let $D\subset \R^d$ be a bounded open set and assume that  $F^{(1)}, F^{(2)}$
     are two non-negative symmetric functions on $\R^d\times\R^d$ vanishing on the diagonal such that
     $$
          F^{(1)}(x,y)=F^{(2)}(x,y)
         \quad\text{for all}\quad
         (x,y)\in (D\times\R^d)\cup (\R^d\times D).
     $$
     Then $u$ is (regular) $F^{ (1)}$ -harmonic in $D$ if, and only, if $u$ is (regular) $F^{ (2)}$-harmonic in $D$.
 \end{lemma}
 
  \begin{lemma}\label{l:fuchsian}
 Let $\beta>1$ and $R\geq 1$. Assume that $F:\R^d\times\R^d \to [0,\infty)$ is symmetric, bounded and satisfies \eqref{e:fuchsian}.
 \begin{itemize}
 \item[(a)]\label{l:fuchsian-a}
     Then $F_R$ is symmetric, bounded and satisfies $F_R(x,y)\le C\Phi^R(|x-y|)^{\beta}$ for all $(x,y)\in (B(0,1)^c\times \R^d)\cup (\R^d\times B(0,1)^c)$.
 \item[(b)]\label{l:fuchsian-b}
     For a bounded open set $D\subset B(0,1)^c$ let
     $$
         \widehat{F}_R(x,y)
         =
         \begin{cases}
         F_R(x,y)  & \text{if\ \ } (x,y)\in (D\times \R^d)\cup (\R^d\times D)\\
         0         & \text{otherwise.}
         \end{cases}
     $$
         Then $\widehat{F}_R$ is symmetric, bounded and satisfies $\widehat{F}_R(x,y)\le C\Phi^R(|x-y|)^{\beta}$ for all $x,y\in\R^d$.
\end{itemize}
 \end{lemma}
 
 \begin{proof}
(a) Clearly, $F_R$ is symmetric and bounded. Further, for $|x|\ge 1$ or $|y|\ge 1$, we have 
 \begin{align*}
     F_R(x,y)
     &= F(Rx, R y)
     \le C\frac{\Phi(|Rx-Ry|)^{\beta}}{1+\Phi(|Rx|)^{\beta}+\Phi(|Ry|)^{\beta}}\\
     &=C \frac{\Phi^R(|x-y|)^{\beta}}{\Phi(R)^{-\beta}+\Phi^R(|x|)^{\beta}+\Phi^R(|y|)^{\beta}}
     \le C\Phi^R(|x-y|)^{\beta}.
 \end{align*}
 (b) This immediately follows from (a).
 \end{proof} 
 
 For a Borel set $C\subset \R^d$ let $T_C=\inf\{t>0:\, X_t\in C\}$ be its hitting time. If $0<a<b$, let $V(0,a,b):=\{x\in \R^d:\, a<|x|<b\}$ be the open annulus, and denote by $\overline{V}(0,a,b)$ its closure.
 \begin{lemma}\label{l:infinite-hitting}
     There exists a positive integer $M=M(d,\delta_1, a_1)\ge 2$ such that for every  strictly increasing sequence of positive numbers $(R_n)_{n\ge 1}$ satisfying $\lim_{n\to \infty}R_n=\infty$ it holds that 
     $$
         \P_x\left(\limsup_{n\to \infty}\, \{T_{\overline{V}(0,R_n,M R_n)}<\infty\}\right) =1
         \quad\text{for all\ \ } x\in\R^d.
     $$
 \end{lemma}
Let $V_n:= \overline{V}(0,R_n,M R_n)$. Lemma \ref{l:infinite-hitting} says that $\P_x\left(\{T_{V_n}<\infty\}\ \text{infinitely often}\right)=1$, i.e.~with $\P_x$ probability $1$, the process $X$ visits infinitely many of the sets $V_n$.
\begin{proof}[Proof of Lemma \ref{l:infinite-hitting}]
By \cite[Corollary 2]{G}, there exists $c=c(d)$ such that for all  $0<s\le r/2$ and all $x\in B(0,s)$,
$$
	\P_x(|X_{\tau_{B(0,s)}}|\ge r)\le c\frac{\phi(r^{-2})}{\phi(s^{-2})}\le c a_1^{-1}\left(\frac{s}{r}\right)^{2\delta_1}\, ,
$$
where the second inequality follows from \eqref{e:H}. 
Choose $M\ge 2$ as the smallest integer such that $M\ge (a_1/2c)^{1/2\delta_1}$. Then $\P_x(|X_{\tau_{B(0,s)}}| > Ms)\le 1/2$ implying that for all $s>0$ and all $x\in B(0,s)$,
$$
\P_x \big( X_{\tau_{B(0,s)}}\in \overline{V}(0,s,Ms) \big) \ge \frac12\, .
$$
In particular, for every $n\ge 1$,
$$
\P_x(T_{\cup_{m=1}^{\infty}V_m}<\infty) \ge \P_x(T_{V_n}<\infty)\ge \P_x(X_{\tau_{B(0,R_n)}}\in V_n)\ge \frac12\, , \quad \text{for all }x\in B(0,R_n)\, .
$$
Since the family of balls $(B(0,R_n))_{n\ge 1}$ covers $\R^d$, by using the same argument as in \cite[Proposition 2.5]{HN13}, we see
that $\P_x(T_{\cup_{m=1}^{\infty}V_m}<\infty) =1$ for all $x\in \R^d$.

Let $C_k:=\bigcup_{n\geq k} V_n$. From the conclusion above we see that  $\P_x(T_{C_k}<\infty)=1$ for every $x\in \R^d$ and $k\geq 1$. Obviously,
$$
         \big\{T_{C_k} < \infty\big\}
         = \big\{T_{\bigcup_{n\geq k} V_n} < \infty\big\}
         = \bigcup_{n\geq k}\big\{T_{V_n} < \infty\big\}.
     $$
Since this inclusion holds for all $k\geq 1$, we get
$$
         \bigcap_{k\geq 1}\big\{T_{C_k} < \infty\big\}
	  = \bigcap_{k\geq 1}\bigcup_{n\geq k}\big\{T_{V_n} < \infty\big\}
       = \limsup_{n\to\infty}\big\{T_{V_n} < \infty\big\}.
     $$
Since $\P_x(T_{C_k} < \infty)=1$ we see
\begin{gather*}
         1 =  \P_x\bigg(\bigcap_{k\geq 1}\big\{T_{C_k} < \infty\big\}\bigg)
	  = \P_x\left(\limsup_{n\to\infty}\big\{T_{V_n} < \infty\big\}\right)
         = \P_x\big(\{T_{V_n}<\infty\}\ \text{i.o.}\big).
     \end{gather*}
\end{proof}

Before proving Theorem \ref{t:finite-expectation}, we need to collect several facts that were proved in \cite[Section 3]{SV}. 
\begin{lemma}\label{l:SV-section3}
Let $Y=(Y_t,\P_x)$ be a strong Markov process with values in $\R^d$. For a bounded $F:\R^d\times \R^d\to [0,\infty)$ let $A_t^F:=\sum_{0<s\le t}F(X_{s-}, X_s)$, and assume that $\P_x(A_{\infty}^F<\infty)=1$ for all $x\in \R^d$. Define $u:\R^d\to [0,1]$ by $u(x):=\E_x[e^{-A_{\infty}^F}]$.
\begin{itemize}
	\item[(a)] (\cite[Lemma 3.2]{SV}) $\P_x(\lim_{t\to \infty}u(Y_t)=1)=1$ for all $x\in \R^d$;
	\item[(b)] (\cite[Lemma 3.4]{SV}) For every open $D\subset \R^d$, $u$ is regular $F$-harmonic in $D$ with respect to $Y$,
	\item[(c)] (\cite[Proposition 3.5]{SV}) If  $\inf_{x\in \R^d}u(x)>0$, then $\sup_{x\in \R^d}\E_x[A_{\infty}^F]<\infty$;
	\item[(d)] (\cite[Proposition 3.6]{SV}) Assume that $Y$ is strong Feller and that $\lim_{t\to 0}\sup_{x\in \R^d}\E_x[A_t^F]=0$. Then $u$ is continuous.
\end{itemize}
\end{lemma}

\noindent
\emph{Proof of Theorem \ref{t:finite-expectation}}.
It follows from \eqref{e:fuchsian} and \eqref{e:phi} that $F(x,y)\le 4^\beta C(\Phi(|x-y|)^{\beta}\wedge 1)$, hence $F\in I(4^\beta C, \Phi(\cdot)^{\beta}\wedge 1)$. 
Let $u(x):=\E_x [e^{-A^F_{\infty}}]$.  
By Lemma \ref{l:SV-section3}(b) and (d),  and Lemma \ref{l:kato-class} it follows that $u$ is continuous and regular $F$-harmonic for $X$ in every bounded open set $D\subset \R^d$. Furthermore, by Lemma \ref{l:SV-section3}(a), we have that $\lim_{t\to \infty}u(X_t)=1$ $\P_x$-a.s. In order to prove that $\sup_{x\in \R^d} \E_x[A^F_{\infty}]$ is finite, by Lemma \ref{l:SV-section3}(c) it is enough to show that $\inf_{x\in \R^d} u(x)=c$ for some $c>0$. Note that, since $u$ is continuous, this is equivalent to
\begin{equation}\label{e:finite-expectation}
 \liminf_{|x|\to\infty}u(x)>0.
\end{equation}
    Let $D=V(0,1,2M+1)=\{x\in \R^d:\, 1<|x|<2M+1\}$ and $R\geq 1$. Since $u$ is regular $F$-harmonic in $D_R$ for $X$, we get from Lemma \ref{l:harmonic-scaling-b} that $u_R$ is regular $F_R$-harmonic in $D$ for $X^R$. Define $\wh{F}_R$ as in Lemma \ref{l:fuchsian}(b).  Since $F_R(x,y)=\widehat{F}_R(x,y)$ for all $(x,y)\in (D\times \R^d)\cup (\R^d \times D)$, $u_R$ is by Lemma \ref{l:harmonic-semilocal} also regular $\widehat{F}_R$-harmonic in $D$ for $X^R$. 
     Moreover, by Lemma \ref{l:fuchsian}, $\widehat{F}_R\in I(4^\beta C,\Phi^R(\cdot)^{\beta}\wedge 1)$. 
    Hence it follows from Theorem \ref{t:hi} that with $c=C_5(d, a_1,a_2,\delta_1, \delta_2,\beta,C)>1$ it holds that
    $$
        c^{-1}u_R(y)\le u_R(x) \le c u_R(y)\quad \text{for all\ \ }x,y\in \overline{V}(0,2,2M)\, .
    $$
Since $R\ge 1$ was arbitrary, we conclude that
    \begin{equation}\label{e:hi-scaled}
        c^{-1}u(y)\le u(x)\le c u(y) \quad \text{for all\ \ }x,y\in \overline{V}(0,2R,2RM)\, ,
    \end{equation}
    for all $R\ge 1$.
    
    Suppose that \eqref{e:finite-expectation} does not hold, i.e.~that there exists a sequence $(x_n)_{n\ge 1}$ in $\R^d$ such that $|x_n|\to \infty$ and $\lim_{n\to \infty} u(x_n)=0$. Then there exists an increasing sequence $(k_n)_{n\geq 1}$ such that $x_n\in V_{n}:=\overline{V}(0,2^{k_n}, 2^{{k_n}}M)$ for every $n\ge 1$. By Lemma \ref{l:infinite-hitting}, $X$ hits infinitely many sets $V_{n}$ $\P_x$-a.s.~Hence, for $\P_x$-a.e.~$\omega$ there exists a subsequence $(n_l=n_l(\omega))$  and a sequence of times $(t_l=t_l(\omega))$ such that $X_{t_l}(\omega)\in V_{{n_l}}$. Therefore it follows from \eqref{e:hi-scaled} that
    $$
        c^{-1}u(X_{t_l}(\omega))\le u(x_{n_l}) \le c u(X_{t_l}(\omega)),
    $$
    which implies that $\lim_{l\to \infty}u(X_{t_l}(\omega))=0$. But this is a contradiction with $\lim\limits_{t\to \infty}u(X_t)=1$ $\P_x$-a.s. Therefore, \eqref{e:finite-expectation} holds.   
\qed

\section{Proof of Theorem \ref{t:infinite-expectation}}\label{s:proof-infinite}
In this section we assume that $\phi$ is a complete Bernstein functions satisfying the global weak scaling condition \eqref{e:H}. 
Recall that under this condition $X$ is transient. Then $X$ admits the radially decreasing  Green function $G(x,y)=G(|x-y|)$, $x,y\in \R^d$. By \cite[Lemma 3.2(b)]{KSV14}, there exists $C_6=C_6(d,a_1,a_2,\delta_1, \delta_2)\ge 1$  such that
\begin{equation}\label{e:free-G-estimates}
C_6^{-1}r^{-d}\Phi(r)\le G(r)\le C_6 r^{-d}\Phi(r)\, , \quad r>0\, .
\end{equation}

The invariant $\sigma$-field $\II$ is defined by
$$
\II=\{\Lambda\in \MM:\, \theta_t^{-1}\Lambda=\Lambda \text{ for all }t\ge 0\}.
$$
Since $X$ has transition densities $p(t,x,y)$, the argument at the end of \cite[Section 2]{SV} shows that if $\II$ is trivial under $\P_x$ for some $x\in \R^d$, then it is trivial under $\P_x$ for all $x\in \R^d$.

\medskip
\noindent
\emph{Proof of Theorem \ref{t:infinite-expectation}}:
 Fix $\gamma$ and $\beta$ so that $0<\gamma <1 <\beta$. Since the function $r\mapsto \Phi(r)^{1-\gamma}$ strictly increases from $0$ to $\infty$, we can choose a sequence   $(x_n)_{n\ge 1}$ of points in $\R^d$ such that $|x_n|>2^n$ and $\Phi(|x_n|)^{1-\gamma}=2^{nd}$ for all $n\ge 1$.  Let $r_n=2^{-n}|x_n|+1<|x_n|$.  Consider the family of balls $\{B(x_n,r_n)\}_{n\ge 1}$. 
By \cite[Lemma 2.5]{MV}, \eqref{e:free-G-estimates} and \eqref{e:H}
    \begin{eqnarray*}
        \lefteqn{\P_0(T_{B(x_n,r_n)}<\infty)
         \le  \frac{|x_n|^{-d}\phi(|x_n|^{-2})^{-1}}{r_n^{-d}\phi(r_n^{-2})^{-1}}\le a_2\left(\frac{r_n}{|x_n|}\right)^{d-2\delta_2}}\\
        &=&\left(\frac{2^{-n}|x_n|+1}{|x_n|}\right)^{d-2\delta_2}
        \le (2^{-n}+2^{-n})^{d-2\delta_2}=2^{(1-n)(d-2\delta_2)}.
    \end{eqnarray*}
    Hence,
    $\sum_{n\ge 1}\P_0(T_{B(x_n,r_n)}<\infty) <\infty$, implying by the Borel--Cantelli lemma that $\P_0(\{T_{B(x_n,r_n)}<\infty\} \text{\ i.o.})=0$. Therefore, $X$ hits $\P_0$-a.s.~only finitely many balls $B(x_n,r_n)$. Let $C:=\bigcup_{n\ge 1}B(x_n,r_n)$.

    Define a symmetric bounded function $F:\R^d\times \R^d \to [0,\infty)$ by
    $$
        F(y,z)
        :=\begin{cases}
            \dfrac{\Phi(|y-z|)^{\beta}}{\Phi(|y|)^{\gamma}+\Phi(|z|)^{\gamma}},& y,z\in B(x_n,r_n) \text{\ for some\ } n, \;  |y-z|\le 1\\[\bigskipamount]
            0, & \text{otherwise}.
        \end{cases}
    $$
    Note that $F(y,z)\le \Phi(|y-z|)^{\beta}\wedge 1$ for all $y,z\in \R^d$. Thus, $F\in I(1,\Phi(\cdot)^{\beta}\wedge 1)$.

    Let $A_t^F:=\sum_{s\le t}F(X_{s-},X_s)$, $t\ge 0$. Then $\E_0 [A_t^F]<\infty$ implying that $\P_0(A_t^F<\infty)=1$ for all $t>0$.  Since $X$ visits only finitely many balls $B(x_n, r_n)$, the last exit time from the union $\bigcup_{n\ge 1}B(x_n,r_n)$ is finite, hence $\P_0(A_{\infty}^F<\infty)=1$. Since $\{A_{\infty}^F<\infty\}\in \II$, we conclude  that $\P_x(A_{\infty}^F<\infty)=1$ for all $x\in \R^d$.

    Further,
    \begin{align*}
    \E_x [A_{\infty}^F]
    &=\E_x \left[\sum_{s>0}F(X_{s-},X_s)\right]\\
    &=\E_x \left[\int_0^{\infty}\!\!\! \int_{\R^d} F(X_{s-},z)j(|X_{s-}-z|)\, dz \, ds\right]\\
   &=\E_x \left[\int_0^{\infty} h(X_s)\, ds\right]
    =Gh(x)
    =\int_{\R^d}h(y)G(x,y)\,dy,
    \end{align*}
    where
    $$
        h(y)
         :=\int_{\R^d}F(y,z)j(|y-z|)\, dz
    $$
    If $y\notin C$, then $F(y,\cdot)=0$, implying that $h(y)=0$. Let $y\in B(x_n,r_n-1)$. Then $|y|\le 2|x_n|$ and if $z$ satisfies $|z-y|<1$, then $z\in B(x_n,r_n)$ and also $|z|\le 2|x_n|$. Therefore, by use of \eqref{e:estimate-of-j} in the first line and \eqref{e:H} in the second
\begin{align*}
    h(y)
    &\ge c_1  \int_{z\in B(x_n,r_n), |z-y|\le 1} \frac{\Phi(|y-z|)^{\beta}}{\Phi(|y|)^{\gamma}+\Phi(|z|)^{\gamma}}\, |y-z|^{-d}\Phi(|y-z|)^{-1}\, dz\\
    &\ge  c_2 \int_{|z-y|\le 1} \frac{\Phi(|y-z|)^{\beta-1}|y-z|^{-d}}{\Phi(|x_n|)^{\gamma}}\, dz \ge c_3 \Phi(|x_n|)^{-\gamma}.
\end{align*}
    In the last inequality we have used that $0 < \int_{|z-y|\le 1}\Phi(|y-z|)^{\beta-1}|y-z|^{-d}\, dz <\infty$. Hence, for $|x|\le 1$ we have $|x-y|\leq 4|x_n|$, so
\begin{align*}
    Gh(x)
    &=  \sum_{n\ge 1}\int_{B(x_n,r_n)}  h(y)G(x,y)\, dy\\
    &\ge   \sum_{n\ge 1}\int_{B(x_n,r_n-1)}  h(y)G(|x-y|)\, dy\\
    &\ge  c_4 \sum_{n\ge 1} \Phi(|x_n|)^{-\gamma} \int_{B(x_n, r_n-1)}G(4|x_n|)\, dy \\
    &\ge  c_5 \sum_{n\ge 1} \Phi(|x_n|)^{-\gamma} |4x_n|^{-d}\Phi(4|x_n|)(r_n-1)^d\\
    &\ge  c_6 \sum_{n\ge 1} \Phi(|x_n|)^{1-\gamma}2^{-nd} \\
    &= c_6 \sum_{n\ge 1} 2^{nd}2^{-nd} = \infty.
\end{align*}
This implies that $Gh\equiv \infty$.\qed

\section{Proof of Theorem \ref{t:entropy}}\label{s:proof-entropy}
 
In order to prove Theorem \ref{t:entropy}, we first collect several auxiliary results from the literature. 

Let $j(x,y):=j(|x-y|)$. The Dirichlet form $(\EE, \FF)$ of $X$ is
$$
    \FF=\left\{f\in L^2(\R^d,dx)\,:\, \int_{\R^d}\int_{\R^d} (f(x)-f(y))^2 j(x,y)\, dx \, dy <\infty\right\},
$$
and
$$
    \EE(f,f)=\int_{\R^d}\int_{\R^d} (f(x)-f(y))^2 j(x,y)\, dx \, dy,\quad f\in \FF.
$$

Let $F\in I_2(X)$ such that $\inf_{x,y\in \R^d} F(x,y) >-1$, and let $\widetilde{X}=(\widetilde{X}_t, \MM, \MM_t, \widetilde{\P}_x)$ be the corresponding  purely discontinuous Girsanov transform of $X$. It follows from \cite[Lemma 2.1 and Theorem 2.5]{S} that the semigroup of $\widetilde{X}$ is symmetric (with respect to Lebesgue measure), and that the Dirichlet form $(\widetilde{\EE}, \widetilde{\FF})$ of $\widetilde{X}$ in $L^2(\R^d, dx)$ is given by $\widetilde{\FF}=\FF$, and
$$
    \widetilde{\EE}(f,f)=\int_{\R^d}\int_{\R^d} (f(x)-f(y))^2 (1+F(x,y))j(x,y)\, dx\, dy, \quad f\in \widetilde{\FF}.
$$
We read that the jumping density $\tilde{j}(x,y)$ of $\wt{X}$ is equal to $(1+F(x,y))j(x,y)$. 
Let $\wt{c_1}:=\inf_{x,y\in \R^d}F(x,y)\le \sup_{x,y\in \R^d}F(x,y)=:\wt{c_2}$. Then $(1+\wt{c_1})j(x,y)\le \tilde{j}(x,y)\le (1+\wt{c_2} )j(x,y)$. Hence, there exists $\wt{C}>1$ such that
\begin{equation}\label{e:j-comparison}
\wt{C}^{-1}j(x,y)\le \tilde{j}(x,y)\le \wt{C}j(x,y)\, ,\quad x,y\in \R^d\, .
\end{equation}
Since the killing measure of $\widetilde{X}$ is zero and $1+F(x,y)$ is bounded from below and above by positive constants, we conclude that $\widetilde{X}$ is conservative. It follows from \eqref{e:estimate-of-j} and \eqref{e:j-comparison} that there exists a positive constant $C_7\ge 1$ such that
\begin{align*}
& C_7^{-1}|x-y|^{-d}\Phi(|x-y|)^{-1}\le j(x,y) \le C_7 |x-y|^{-d}\Phi(|x-y|)^{-1}\, , \quad x,y\in \R^d\, ,\\
& C_7^{-1}|x-y|^{-d}\Phi(|x-y|)^{-1}\le \tilde{j}(x,y) \le C_7 |x-y|^{-d}\Phi(|x-y|)^{-1}\, , \quad x,y\in \R^d\, .
\end{align*}
It is now straightforward to check that the conditions of \cite[Theorem 1.2]{CK08} are satisfied for both $j(x,y)$ and $\tilde{j}(x,y)$ (with $\varphi_1$ from \cite{CK08} equal to our $\Phi$ and $\psi\equiv 1$). Hence, it follows from \cite[Theorem 1.2]{CK08} that $\wt{X}$ has transition densities $\wt{p}(t,x,y)$ satisfying the sharp two-sided estimates: For all $t>0$ and all $x,y\in \R^d$,
$$
C_8^{-1}\left(\frac{1}{\Phi^{-1}(t)^d}\wedge \frac{t}{|x-y|^d\Phi(c|x-y|)}\right)\le \wt{p}(t,x,y) \le C_8 \left(\frac{1}{\Phi^{-1}(t)^d}\wedge \frac{t}{|x-y|^d\Phi(c|x-y|)}\right)
$$
with $C_8\ge 1$ and $c>0$. Here $\Phi^{-1}$ is the inverse of the strictly increasing function $\Phi$. We note that the fact that the same constant $c$ appears on both sides of the estimate is a consequence of the scaling of $\Phi$. The same sharp two-sided estimates are valid for transition densities $p(t,x,y)$ of the process $X$. As a consequence, with $C_9=C_8^2$ we have that
\begin{equation}\label{e:comparison-p}
C_9^{-1}p(t,x,y) \le \wt{p}(t,x,y) \le C_9 p(t,x,y)\, ,\quad t>0, x,y\in \R^d\, .
\end{equation}
By integrating over $(0,\infty)$ with respect to time $t$ we arrive at
\begin{equation}\label{e:comparison-G}
C_9^{-1}G(x,y) \le \wt{G}(x,y) \le C_9 G(x,y)\, ,\quad x,y\in \R^d\, ,
\end{equation}
where $\wt{G}(x,y)$ denotes the Green function of the process $\wt{X}$.

It follows from \cite[Lemma A.1]{MV} that if $h:\R^d \to [0,\infty)$ is harmonic in $\R^d$ with respect to $X$,  then $h$ is a constant function. This implies that the Martin compactification of $\R^d$ consists of a single point. This is equivalent to the fact that
\begin{equation}\label{e:Limit-at-infty}
\lim_{|y|\to \infty}\frac{G(x,y)}{G(0,y)}=1\, .
\end{equation}

Recall that the invariant $\sigma$-field was defined by $\II=\{\Lambda\in \MM:\, \theta_t^{-1}\Lambda=\Lambda \text{ for all }t\ge 0\}$. With \eqref{e:Limit-at-infty} and \eqref{e:comparison-G} at hand, the proof of the next lemma is exactly the same as the one of \cite[Lemma 5.1]{SV}.

\begin{lemma}\label{l:invariant-trivial}
    The invariant $\sigma$-field $\II$ is trivial with respect to both $\P_x$ and $\widetilde{\P}_x$ for all $x\in \R^d$.
\end{lemma}

\medskip
\noindent
\emph{Proof of Theorem \ref{t:entropy}}: (a) Dichotomy  $\wt{\P}_x\perp \P_x$ or $\wt{\P}_x\sim \P_x$ follows directly from \cite[Corollary 2.13]{SV} and Lemma  \ref{l:invariant-trivial}. Further, the assumption $\wt{\P}_x\sim \P_x$ implies by \cite[Theorem 1.1]{SV} that
$$
\P_x\left(\sum_{t>0}F^2(X_{t-}, X_t)<\infty\right)= 1\, .
$$
Next, by the assumption \eqref{e:fuchsian-squared-a},
$$
F^2(x,y)\le C^2 \frac{\Phi(|x-y|)^{2\beta}}{(1+\Phi(|x|)^{\beta}+\Phi(|y|)^{\beta})^2}\le C^2\frac{\Phi(|x-y|)^{2\beta}}{1+\Phi(|x|)^{2\beta}+\Phi(|y|)^{2\beta}}\, ,
$$
implying that $F^2$ satisfies \eqref{e:fuchsian} with $2\beta >1$. By use of Theorem \ref{t:finite-expectation} we conclude that $\sup_{x\in \R^d} \E_x \left[\sum_{t>0} F^2 (X_{t-}, X_t)\right] <\infty$. Finally, \cite[Theorem 1.1(\~c)]{SV}   implies that $\sup_{x\in \R^d} \HH(\P_x; \widetilde{\P}_x)<\infty$.

\noindent 
(b) By assumption $2\gamma <1 <2\beta$. 
Denote by $B(x_n,r_n)$ the balls constructed in the proof of Theorem \ref{t:infinite-expectation} and let
    $$
        H(x,y)
        =
        \begin{cases}
            \dfrac{\Phi(|x-y|)^{2\beta}}{\Phi(|x|)^{2\gamma}+\Phi(|y|)^{2\gamma}}, & x,y\in B(x_n,r_n) \text{\ for some $n$ and $|x-y|<1$},\\[\bigskipamount]
            0, & \text{otherwise}.
        \end{cases}
    $$
    Define $F(x,y)=\frac18 \sqrt{H(x,y)}$. Then
    $$
        F(x,y)
         = \frac18\, \frac{\Phi(|x-y|)^{\beta}}{\sqrt{\Phi(|x|)^{2\gamma}+\Phi(|y|)^{2\gamma}}}
        \le \frac14 \frac{\Phi(|x-y|)^{\beta}}{\Phi(|x|)^{\gamma}+\Phi(|y|)^{\gamma}}
        \le \frac12 \frac{\Phi(|x-y|)^{\beta}}{1+\Phi(|x|)^{\gamma}+\Phi(|y|)^{\gamma}}
    $$
since we can take $|x|,|y|\ge 1$.
    As $\sum_{t>0}F^2(X_{t-}, X_t)=\frac{1}{64} \sum_{t>0}H(X_{t-},X_t)$, 
    we see from Theorem \ref{t:infinite-expectation} that $\sum_{t>0}F^2(X_{t-}, X_t)<\infty$ $\P_x$ a.s.~and $\E_x[\sum_{t>0}F^2(X_{t-}, X_t)]=\infty$. By \cite[Theorem 1.1(\~b) and (\~c)]{SV} it follows that $\P_x\ll \widetilde{\P}_x$ and $\H(\P_x;\widetilde{\P}_x)=\infty$.
\qed

\bigskip
\noindent
{\bf Acknowledgements:} 
We thank the referee for very helpful comments on the first version of this paper.

\end{doublespace}

\bigskip
\noindent
\vspace{.1in}
\begin{singlespace}


\end{singlespace}

\vskip 0.1truein

\parindent=0em

\bigskip

{\bf Zoran Vondra\v{c}ek}

Department of Mathematics, Faculty of Science, University of Zagreb, Zagreb, Croatia,

Email: \texttt{vondra@math.hr}

\bigskip

{\bf Vanja Wagner}

Department of Mathematics, Faculty of Science, University of Zagreb, Zagreb, Croatia,

Email: \texttt{vanja.wagner@math.hr}


\small
\begin{thebibliography}{99}
\bibitem{B-AP}
    I.~Ben-Ari, R.G.~Pinsky:
    Absolute continuity/singularity and relative entropy properties for probability measures induced by diffusions on infintie time intervals. \emph{Stoch.~Proc.~Appl.} \textbf{115} (2005) 179--206.

\bibitem{CK08}
	Z.Q.~Chen, T.~Kumagai:
	Heat kernel estimates for jump processes of mixed types on metric measure spaces.
	 \emph{Probab.~Theory Rel.~Fields} \textbf{140} (2008) 277--317.

\bibitem{CS03}
    Z.-Q.~Chen, R.~Song:
    Conditional gauge theorem for non-local Feynman-Kac transforms.
    \emph{Probab.~Theory Rel.~Fields} \textbf{125} (2003) 45--72.
    
\bibitem{CS03b}
    Z.-Q.~Chen, R.~Song:
    Drift transforms and Green function estimates for discontinuous processes.
    \emph{J.~Funct.~Anal.} \textbf{201} (2003) 262--281.
    
\bibitem{G}
 	T.~Grzywny:
 	On Harnack inequality and H\"older regularity for isotropic unimodal L\'evy processes.
 	\emph{Potential Anal.} \textbf{41} (2014) 1--29.
 
 \bibitem{HN13}
     W.~Hansen, I.~Netuka:
     Unavoidable sets and harmonic measures living on small sets.
     \emph{Proc.~London Math.~Soc.} \textbf{109} (2014) 1601--1629.
     
\bibitem{KM}
      P.~Kim, A.~Mimica:
      Green function estimates for subordinate Brownian motions: Stable and beyond.
      \emph{Trans.~Amer.~Math.~Soc.} \textbf{366(8)} (2014) 4383--4422.
      
\bibitem{KSV14}
    P.~Kim, R.~Song, Z.~Vondra\v cek:
    Global uniform boundary Harnack principle with explicit decay rate and its application.
    \emph{Stochastic Process.~Appl.} \textbf{124} (2014) 235--267.
    
 \bibitem{KSV15}
    P.~Kim, R.~Song, Z.~Vondra\v cek:
    Martin boundary for some symmetric Levy processes. 
    \emph{In: Festschrift Masatoshi Fukushima, Eds. Z.-Q.Chen, N.Jacob, M.Takeda, T.Uemura, World Scientific} (2015) 307--342. 
    
 \bibitem{KSV16}
    P.~Kim, R.~Song, Z.~Vondra\v cek:
    Potential theory of subordinate killed Brownian motion, arXiv:1610.00872 (2016) 54 pp.
    
\bibitem{MV}
    A.~Mimica, Z.~Vondra\v cek:
    Unavoidable collections of balls for isotropic L\'evy processes.
    \emph{Stochastic Process.~Appl.} \textbf{124} (2014) 1303--1334.
    
\bibitem{Nab}
	D.~Nabeshima:
	Absolute continuity and relative entropy of probability measures induced by a certain Girsanov transform.
	Master Thesis, Kumamoto University, 2016.
    
 \bibitem{Sh}
     M.~Sharpe:
    \emph{General Theory of Markov Processes}.
     Academic Press, Boston 1988.
		
\bibitem{SSV}
    R.~Schilling, R.~Song  and Z.~Vondra\v{c}ek. \emph{Bernstein Functions: Theory and Applications}. 2nd ed., de Gruyter, 2012.
 
\bibitem{SV}
	R.L.~Schilling, Z.~Vondra\v cek: 
	Absolute continuity and singularity of probability measures induced by a purely discontinuous Girsanov transform of a stable process.
	\emph{Trans.~Amer.~Math.~Soc.} \textbf{369(3)} (2017) 1547--1577.
	
\bibitem{S}
     R.~Song:
     Estimates on the transition densities of Girsanov transforms of symmetric stable processes.
     \emph{J.~Theor.~Probab.} \textbf{19} (2006) 487--507.
	

\end{thebibliography}
\end{document}